\documentclass{amsart}
\RequirePackage{amsmath}
\RequirePackage{amssymb}
\RequirePackage{amsxtra}
\RequirePackage{amsfonts}
\RequirePackage{latexsym}
\RequirePackage{euscript}
\RequirePackage{amscd}
\RequirePackage{amsthm}
\RequirePackage{xypic}
\xyoption{all}

\DeclareMathAlphabet{\mathpzc}{OT1}{pzc}{m}{it}

\DeclareMathOperator{\End}{End}
\DeclareMathOperator{\Hom}{Hom}
\DeclareMathOperator{\Ind}{Ind}

\DeclareMathOperator{\GL}{GL}

\DeclareMathOperator{\Mod}{Mod}

\DeclareMathOperator{\val}{val}
\DeclareMathOperator{\Inj}{Inj}

\DeclareMathOperator{\Ord}{Ord}

\DeclareMathOperator{\Irr}{Irr}

\newcommand{\Indu}[3]{\Ind_{#1}^{#2}{#3}}

\newcommand{\oo}{\mathfrak o}
\newcommand{\oF}{\oo}

\newcommand{\Z}{{\mathbb Z}}

\newcommand{\mm}{\mathfrak{m}}

\newcommand{\Q}{{\mathbb Q}}

\newcommand{\Qp}{\mathbb {Q}_p}

\newcommand{\Fpbar}{\overline{\mathbb{F}}_p}
\newcommand{\iso}{\buildrel \sim \over \longrightarrow}
\newcommand{\unif}{\varpi}

\def\Mod{\mathrm{Mod}}

\def\smGmod{\Mod^{\mathrm{sm}}_G}
\def\admGmod{\Mod^{\mathrm{adm}}_G}

\def\locadmGmod{\Mod^{\mathrm{l.adm}}_G}

\def\locadmGmodz{\Mod^{\mathrm{l.adm}}_{G,\zeta}}

\def\admGmodz{\Mod^{\mathrm{adm}}_{G,\zeta}}

\def\smGmodz{\Mod^{\mathrm{sm}}_{G,\zeta}}

\def\smHmod{\Mod^{\mathrm{sm}}_H}

\def\locadmMmod{\Mod^{\mathrm{l.adm}}_M}

\def\smKmod{\Mod^{\mathrm{sm}}_K}
\def\smKmodz{\Mod^{\mathrm{sm}}_{K,\zeta}}

\swapnumbers

\newtheorem{lemma}[subsection]{Lemma}
\newtheorem{df}[subsection]{Definition}
\newtheorem{cor}[subsection]{Corollary}

\newtheorem{prop}[subsection]{Proposition}
\newtheorem{remark}[subsection]{Remark}

\newtheorem{thm}[subsection]{Theorem}

\newtheorem{remar}[subsection]{Remark}

\title[Ordinary parts of admissible representations II]
{Ordinary parts of admissible representations of 
$p$-adic reductive groups II.  Derived functors }
\title{On the effaceability of certain $\delta$-functors}
\author{Matthew Emerton \and Vytautas Pa\v{s}k\={u}nas}
\thanks{The first author was supported in part by the NSF grant 
DMS-0701315}
\address[Matthew Emerton]{Mathematics Department, Northwestern
University, 2033 Sheridan Rd., Evanston, IL 60208}
\email[Matthew Emerton]{emerton@math.northwestern.edu}

\address[Vytautas Pa\v{s}k\={u}nas]{Fakult\"at f\"{u}r Mathematik, Universit\"at Bielefeld,
Postfach 100131, D-33501 Bielefeld}
\email[Vytautas Pa\v{s}k\={u}nas]{paskunas@math.uni-bielefeld.de}
\begin{document}

\maketitle
%\abstract{We prove a (variant of) conjecture of Emerton for $\GL_2(F)$, where $F$ is a finite extension of $\Qp$.}
\section{Introduction}
\parindent0mm
\parskip5pt
Let $F$ be a finite extension of $\Qp$ and let $\oF$ be its ring of integers. Let $G:=\GL_2(F)$,
let $K:=\GL_2(\oF)$, and let $Z$ be the centre of $G$.
Let $A$ be a finite local Artinian $\Z_p$-algebra with residue field $k$
(necessarily finite, of characteristic $p$).
Recall that 
a representation $V$ of $G$ on an $A$-module is said to be \textit{smooth}
if for all $v\in V$ the stabilizer of $v$ is an open subgroup of $G$. 
Let $\smGmod(A)$  denote the category of smooth $A$-representations. Further recall that a smooth $A$-representation $V$ is \textit{admissible} 
if for every open subgroup $J$ of $G$ the space $V^J$ of $J$-invariants is a finite $A$-module. Let $\admGmod(A)$ denote the full subcategory 
of $\smGmod(A)$ consisting of admissible representations.
The categories $\admGmod(A)$ and $\smGmod(A)$ are abelian. In practice, one
is interested in admissible representations, but $\admGmod(A)$ does not have enough injectives. The category $\smGmod(A)$
has enough injectives, but it is too big. To remedy this the first author,
in \cite{em1}, \cite{em2},
has introduced an intermediate category 
of locally admissible representations $\locadmGmod(A)$.
We recall the definition: If $V$ is a smooth $A$-representation of $G$,
a vector $v\in V$ is 
called \textit{locally admissible} if the $A[G]$-submodule of $V$ generated by $v$
is admissible; 
a smooth representation $V$ of $G$ over $A$ is then called \textit{locally admissible}
if every $v\in V$ is locally admissible. We let $\locadmGmod(A)$ denote the full 
subcategory of $\smGmod(A)$ consisting of locally admissible representations.
The category $\locadmGmod(A)$ is abelian and has enough injectives 
\cite[Prop.~2.2.15]{em1}, \cite[Prop.~2.1.1]{em2}.

We introduce some variants of the preceding categories:

If $\zeta: Z\rightarrow A^{\times}$ is a smooth character,
then we denote by 
$\admGmodz(A)$, 
$\locadmGmodz(A)$, and $\smGmodz(A)$
the full subcategories of
$\admGmod(A)$,
$\locadmGmod(A)$, and $\smGmod(A)$ respectively,
consisting of representations admitting $\zeta$ as a central character.
We also let $\smKmodz(A)$ denote the full subcategory 
of $\smKmod(A)$ consisting of $K$-representations admitting 
$\zeta_{| Z \cap K}$ as a central character.
The categories
$\admGmodz(A)$, $\locadmGmodz(A)$, $\smGmodz(A)$,
and $\smKmodz(A)$
are abelian, and the last three have enough injectives. 
(See Lemma~\ref{lem:z categories} below.)

%Let $\varpi$ be a uniformizer of $F$,
%and view $\varpi$ as an element of $Z$ via the isomorphism $Z\cong F^{\times}$.
%If $u\in A^{\times}$, 
%then we let 
%$\admGmodu(A)$, $\locadmGmodu(A)$, and $\smGmodu(A)$ 
%denote the full subcategories of
%$\admGmod(A)$,
%$\locadmGmod(A)$, and $\smGmod(A)$ respectively,
%consisting of representations on which $\varpi$ acts via $u$.
%Again these categories are abelian and the last two have enough injectives.
%(See Lemma~\ref{lem:z categories}~(2) below.)

In this note we show that the restriction to $K$ of an  injective object in
$\locadmGmodz(A)$ (resp.\ %$\locadmGmodu(A)$ or 
$\locadmGmod(A)$)
is an injective object in  $\smKmodz(A)$ (resp. $\smKmod(A)$).
%(Here $\smKmodz(A)$ has the obvious meaning, namely it is the full subcategory 
%of $\smKmod(A)$ consisting of $K$-representations admitting 
%$\zeta_{| Z \cap K}$ as a central character.)
This implies that certain $\delta$-functors defined in \cite{em2} are effaceable,
and remain effaceable when restricted to $\locadmGmodz(A)$. %or $\smGmodu(A)$.
In particular, it proves Conjecture~3.7.2 of \cite{em2} for $\GL_2(F)$.

\textit{Acknowledgments.} The authors
would like to thank Florian Herzig for useful comments, which have improved the exposition. The second author's work on this note was undertaken while
he was visiting Universit\'e Paris-Sud, supported by Deutsche Forschungsgemeinschaft, and
he would like to thank these institutions. 

\section{Injectives}
We establish some simple results about injective objects in various contexts.
In this section we change our notational conventions from those of the introduction,
and let $G$ denote an arbitrary $p$-adic analytic group.

\begin{lemma}
\label{lem:essential}
If $G$ is compact,
if $V$ is an injective object of $\smGmod(k)$, and if $W$ is an
injective envelope of $V$ in $\smGmod(A)$, then the inclusion
$V \hookrightarrow W$ induces an isomorphism $V\iso W[\mathfrak m]$.
\end{lemma}
\begin{proof}
Certainly the inclusion $V \hookrightarrow W$ factors through an inclusion
$V \hookrightarrow W[\mathfrak m]$.  Since the source is injective,
this inclusion splits.  If $C$ denotes a complement to the inclusion,
then $V \cap C = 0,$ and thus $C = 0$ (as $W$ is an essential extension
of $V$).  This proves the lemma.
\end{proof}

\begin{lemma}
\label{lem:adm}
Let $H$ be a finite index open subgroup of $G$.
\begin{enumerate}
\item[(i)]
An object of $\smGmod(A)$ is admissible {\em (}resp.\ locally admissible{\em )}
as a $G$-representation if and only if it is so as an $H$-representation.
\item[(ii)]
If $V$ is an object of $\smHmod(A)$, so that $\Ind_H^G V$ {\em (}$\iso 
A[G]\otimes_{A[H]} V${\em )} is an object of $\smGmod(A)$,
then $\Ind_H^G V$ is admissible {\em (}resp.\ locally admissible{\em )} as 
a $G$-representation
if and only if $V$ is admissible {\em (}resp.\ locally admissible{\em )}
as an $H$-representation.
\end{enumerate}
\end{lemma}
\begin{proof}
The admissibility claim of part~(i) is clear, since $H$ contains a cofinal
collection of open subgroups of $G$.  Since $H$ has finite index in $G$,
the group ring $A[G]$ is finitely generated as an $A[H]$-module, and thus an $A[G]$-module
is finitely generated if and only if it is finitely generated as an $A[H]$-module.
The local admissibility claim of part~(i) follows from this, together with the
admissibility claim, since an $A[G]$-module
(resp.\ $A[H]$-module) is locally admissible if and only if every finitely generated
submodule is admissible.

To prove the if direction of claim~(ii), suppose first that $V$ is an admissible
$H$-representation.  If we write $G$ as a union of finitely many left $H$-cosets,
say $G = \bigcup_{i=1^n g_i H}$, if $H'$ is an open subgroup of $H$,
and if we write $H'' := H' \cap \bigcap_{i=1}^n g_i H g_i^{-1},$
then
\begin{multline*}
(\Ind_H^G V)^{H'} \subset (\Ind_H^G V)^{H''} \iso (A[G]\otimes_{A[H]} V)^{H''} 
\\
\iso \oplus_{i=1}^n (g_i V)^{H''} = \oplus_{i = 1}^n g_i V^{g_i^{-1} H'' g_i}.
\end{multline*}
Since $g_i^{-1} H'' g_i$ is an open subgroup of $H$, each of the summands 
appearing on the right-hand side is a finite $A$-module, and thus so is their direct sum.
Thus $\Ind_H^G V$ is admissible as claimed.  If we suppose that $V$ instead is
locally admissible, or equivalently, is the inductive limit of its admissible
subrepresentations, we see that the same is true of $\Ind_H^G V$, since
$\Ind_H^G$ commutes with the formation of induction limits (being naturally
isomorphic to $A[G]\otimes_{A[H]}\text{--}$).

To prove the other direction of~(ii),
note first that the inclusion $A[H] \subset A[G]$ gives rise to
an $H$-equivariant embedding $V \hookrightarrow A[G]\otimes_{A[H]} V \iso
\Ind_H^G V.$  Thus if $\Ind_H^G V$ is (locally) admissible as a $G$-representation,
and hence also (locally) admissible as an $H$-representation, by part~(i), the
same is true of its $H$-subrepresentation $V$.
\end{proof}

%\begin{prop}
%\label{prop:finite open}
%If $H$ is an open subgroup of $G$ of finite index,
%then any injective object of
%$\admGmod(A)$ {\em (}resp.\ $\locadmGmod(A)${\em )}
%is also injective as 
%an object of
%$\admHmod(A)$ {\em (}resp.\ $\locadmHmod(A)${\em )}.
%\end{prop}
%\begin{proof}
%This is proved in an identical manner to Proposition~\ref{prop:open},
%taking into account that $c-\Ind_H^G = \Ind_H^G$ and the forgetful
%functor from $G$-representations to $H$-representations both preserve admissibility
%and local admissibility, by Lemma~\ref{lem:adm}.
%\end{proof}

\begin{prop}
\label{prop:doubling}
If $H$ is an open subgroup of $G$ of finite index,
then an object $V$ 
of the category $\smGmod(A)$ 
{\em (}resp.\ $\admGmod(A)$, resp.\ $\locadmGmod(A)${\em )}
is injective 
if and only if
$\Ind_H^G V$ is injective as an object of the same category.
\end{prop}
\begin{proof}
Consider the sequence of adjunction isomorphisms
$$\Hom_{A[G]}(U,\Ind_H^G V) \iso \Hom_{A[H]}(U,V) \iso \Hom_{A[G]}(\Ind_H^G U,V).$$
Since the composite 
of $\Ind_H^G$ (which is naturally equivalent to $A[G]\otimes_{A[H]}\text{--}$)
and the forgetful functor induces an exact functor from
$\smGmod(A)$
(resp.\ $\admGmod(A)$, resp.\ $\locadmGmod(A)$)
to itself
(here we are taking into account Lemma~\ref{lem:adm}),
the proposition follows.
\end{proof}

%\begin{remark}
%{\em Although we will only apply the preceding results in the case
%when indeed $G = \GL_2(F)$, we note for possible future applications
%that the proofs are valid with $G$ being any $p$-adic analytic group.
%}
%\end{remark}

\begin{df}
{\em
%\begin{enumerate}
%\item
If $Z$ denotes the centre of $G$,
if $\zeta: Z \rightarrow A^{\times}$ is a smooth character
and $V$ is a representation of $G$ %(resp.\ $K$)
over $A$,
then we let 
$$V^{Z = \zeta} := \{ v \in V \, | \, z\cdot v = \zeta(z) v \text{ for all } z \in Z\}.$$
%(resp.\ $V^{Z\cap K = \zeta}
%:= \{ v \in V \, | \, z\cdot v = \zeta(z) v \text{ for all } z \in Z\cap K\}.$)
%\smallskip
%\item
%If $u \in A^{\times}$ and $V$ is a representation of $G$ over $A$,
%then we let
%$$V^{\unif = u} := \{ v \in V \, | \, \unif \cdot v = u v \}.$$
%\end{enumerate}
}
\end{df}
Since the subrepresentation of a smooth admissible
(resp.\ smooth locally admissible, resp.\ smooth)
representation is again smooth admissible
(resp.\ smooth locally admissible, resp.\ smooth),
we see, in the context of the preceding definition, that
the construction $V \mapsto V^{Z = \zeta}$ induces a functor 
$\admGmod(A) \rightarrow \admGmodz(A)$
(resp.\
$\locadmGmod(A) \rightarrow \locadmGmodz(A)$,
resp.\ 
$\smGmod(A) \rightarrow \smGmodz(A)$)
that is right adjoint to the forgetful functor.
%that the construction
%$V \mapsto V^{Z\cap K = \zeta}$ induces
%a functor $\smKmod(A) \rightarrow \smKmodz(A)$ that is right adjoint to
%the forgetful functor,
%and that the construction
%$V \mapsto V^{\unif = u}$ induces a functor 
%$\admGmod(A) \rightarrow \admGmodu(A)$
%(resp.\
%$\locadmGmod(A) \rightarrow \locadmGmodu(A)$,
%resp.\ 
%$\smGmod(A) \rightarrow \smGmodu(A)$)
%that is right adjoint to the forgetful functor.
In particular, the functor $V\mapsto V^{Z = \zeta}$
%$V \mapsto V^{Z \cap K = \zeta}$,
%and $V \mapsto V^{\unif = u}$ all
preserves injectives.

\begin{lemma}
\label{lem:z categories}
%\begin{enumerate}
%\item
If $\zeta:Z \rightarrow A^{\times}$ is a smooth character, then
each of the categories 
$\admGmodz(A)$, 
$\locadmGmodz(A)$, and $\smGmodz(A)$
%$\smKmodz(A)$
are abelian, and the last two have enough injectives. 
%\item
%If $u \in A^{\times}$, then each of the categories
%$\admGmodu(A)$, $\locadmGmodu(A)$, and $\smGmodu(A)$
%is abelian, and the last two have enough injectives. 
%\end{enumerate}
\end{lemma}
\begin{proof}
The abelianess claims are evident.  To establish the claim %of~(1)
regarding injectives,
let $V$ be an object of
$\locadmGmodz(A)$ (resp.\ $\smGmodz(A)$)
 and let $V \hookrightarrow W$
be an $A[G]$-linear embedding of $V$ into an injective object in
$\locadmGmod(A)$ (resp.\ $\smGmod(A)$).
This embedding then factors through an embedding $V \hookrightarrow W^{Z = \zeta}$,
and the latter object is injective in
$\locadmGmodz(A)$ (resp.\ $\smGmodz(A)$), 
as was noted above.
\end{proof}

\begin{lemma}\label{centr} Let $G$ be a compact $p$-adic analytic group, let $H$ be a closed subgroup 
containing the centre of $G$ and let $\zeta:Z\rightarrow A^{\times}$ be a smooth character. If $V$ is injective in 
$\Mod^{\mathrm{sm}}_{G, \zeta}(A)$, then it is also injective in $\Mod^{\mathrm{sm}}_{H, \zeta}(A)$.
\end{lemma}
\begin{proof} Let $\iota: V\hookrightarrow J$ be an injective envelope of $V$ in $\Mod^{\mathrm{sm}}_{G}(A)$. 
Since $V$ is injective in $\Mod^{\mathrm{sm}}_{G, \zeta}(A)$ and $\iota$ is essential we deduce that 
$\iota(V)=J^{Z=\zeta}$. Proposition 2.1.11 in \cite{em2}  implies that $J$ is injective in $\Mod^{\mathrm{sm}}_{H}(A)$ 
and thus $J^{Z=\zeta}$ is injective in $\Mod^{\mathrm{sm}}_{H, \zeta}(A)$.
\end{proof}

\section{Main result}
We introduce notation for some subgroups of $G:=\GL_2(F)$ that we will need to consider,
namely:
we write
$G^+:=\{g\in G: \val_F(\det g)\equiv 0\pmod{2}\}$ and $G^0:=\{g\in G: \val_F (\det g)=0\}$,
write $I := \left(\smallmatrix \oF^{\times} & \oF \\ \unif \oF & \oF^{\times} 
\endsmallmatrix\right)$
(an Iwahori subgroup of $K$) and let $I_1$ denote the maximal pro-$p$ subgroup of $I$, 
let
$N_G(I)$ denote the normalizer in $G$ of $I$, 
set $\Pi:=\left(\smallmatrix 0 & 1 \\ \unif & 0 \endsmallmatrix\right)\in N_G(I)$,
%set
%$\mathcal G := N_G(I)/\unif^{\Z},$
and write
$N_0 := \left(\smallmatrix 1 & \oF \\ 0 & 1 \endsmallmatrix\right).$
%\{ \left(\smallmatrix 1 & x \\ 0 & 1 \endsmallmatrix\right)
%\, | \, x \in \oF \}.$

\begin{lemma}\label{A1} If $\iota:V\hookrightarrow J$ is an injective envelope of $V$ in $\Mod^{\mathrm{sm}}_I(A)$, then any isomorphism 
$\psi: V\overset{\cong}{\rightarrow} V^{\Pi}$ extends to an isomorphism $J\cong J^{\Pi}$. 
\end{lemma}
\begin{proof} Since $\iota^{\Pi}: V^{\Pi}\hookrightarrow J^{\Pi}$ is an injective envelope of $V^{\Pi}$ in $\Mod^{\mathrm{sm}}_I(A)$,
the assertion follows from the fact that injective envelopes are unique up to isomorphism.
\end{proof}

\begin{lemma}\label{A2} For an injective admissible object   $J$ in $\Mod^{\mathrm{sm}}_I(A)$ the following are equivalent:
\begin{itemize}
\item[(i)] $J\cong J^{\Pi}$;
\item[(ii)] $ J[\mm]^{I_1}\cong (J[\mm]^{I_1})^{\Pi}$;
\item[(iii)] $\dim_k \Hom_I( \chi, J[\mm]^{I_1})= \dim_k \Hom_I(\chi^{\Pi}, J[\mm]^{I_1})$, $\forall \chi\in \Irr_I(k)$. 
\end{itemize}
\end{lemma}
\begin{proof} Since $J[\mm]^{I_1}\hookrightarrow J$ is essential the equivalence of (i) and (ii) follows from Lemma~\ref{A1}.
Since $J$ is admissible $J[\mm]^{I_1}$ is a finite dimensional $k$-vector space. Since the order of $I/I_1$ is prime to $p$ 
we may write $J[\mm]^{I_1}\cong \oplus_{\chi\in \Irr_I(k)} \chi^{\oplus m_{\chi}}$ and thus $J[\mm]^{I_1}\cong (J[\mm]^{I_1})^{\Pi}$ 
if and only if $m_{\chi}=m_{\chi^{\Pi}}$. Hence, (ii) is equivalent to (iii).
\end{proof} 

\begin{lemma}\label{A3} If $J$ is an admissible injective object in $\Mod^{\mathrm{sm}}_K(A)$, then 
$$\dim_k \Hom_I( \chi, J[\mm]^{I_1})= \dim_k \Hom_I(\chi^{\Pi}, J[\mm]^{I_1}), \quad \forall \chi\in \Irr_I(k).$$ 
\end{lemma}
\begin{proof} Since $J[\mm]$ is injective in $\Mod^{\mathrm{sm}}_K(k)$ we may assume that $A=k$ so that $J[\mm]=J$. 
Further, it is enough to prove the statement for $J=\Inj \sigma$ an injective envelope of an 
irreducible $K$-representation $\sigma$, since any admissible injective object of $\Mod^{\mathrm{sm}}_K(A)$ is 
isomorphic to a finite direct sum of such representations. If $k=\Fpbar$ then the assertion for $J=\Inj \sigma$  follows 
from \cite[Lem. 6.4.1, 4.2.19, 4.2.20]{coeff}, see also the proof of \cite[Lem. 9.6]{bp}. (It is enough 
to assume that $k$ contains the residue field of $F$, in which case every irreducible $k$-representation 
of $K$ or $I$ is absolutely irreducible.) The result for general $k$ follows by Galois descent.   
\end{proof}

\begin{thm}\label{input} If $V$ is an object in $\Mod^{\mathrm{adm}}_{G^0}(A)$ such that $V\cong V^{\Pi}$,
then there exists a $G^0$-equivariant injection $V\hookrightarrow \Omega$ in  $\Mod^{\mathrm{adm}}_{G^0}(A)$ such that 
$V|_{K}\hookrightarrow \Omega|_{K}$ 
is an injective envelope of $V|_{K}$ in $\Mod^{\mathrm{sm}}_K(A)$.
\end{thm}
\begin{proof} The proof is a variation on constructions of \cite{bp} and  \cite{Pas}. It relies 
on the fact that $G^0$ is an amalgam of $K$ and $K^{\Pi}$ along $I=K\cap K^{\Pi}$. Let 
$\iota_0: V|_{K} \hookrightarrow J_0$ be an injective envelope of $V$ in $\Mod^{\mathrm{sm}}_K(A)$
and let $\iota_1: V|_{K^{\Pi}} \hookrightarrow J_1$ be an injective envelope of $V$ in $\Mod^{\mathrm{sm}}_{K^{\Pi}}(A)$. We claim that 
there exists an $I$-equivariant isomorphism $\varphi: J_0\overset{\cong}{\rightarrow} J_1$ such that the diagram 
\begin{displaymath}
\xymatrix@1{ \;V\;\ar[d]_{=}\ar@{^(->}[r]^-{\iota_0} & J_0\ar[d]^{\varphi}_{\cong}\\
              \;V\; \ar@{^{(}->}[r]^-{\iota_1} &  J_1.}
\end{displaymath}
commutes. Granting the claim we may using $\varphi$ transport the action of $K^{\Pi}$ on $J_0$ such that 
the two actions of $I$ on $J_0$ via embeddings $I\hookrightarrow K$, $I\hookrightarrow K^{\Pi}$ coincide.  
Since $G^0$ is an amalgam of $K$ and $K^{\Pi}$ along $I=K\cap K^{\Pi}$ we obtain an action of $G^0$ on $J_0$ and since the diagram 
is commutative $\iota_0: V\hookrightarrow J_0$ is $G^0$-equivariant. 

To prove the claim we closely follow the proof of Theorem 9.8 \cite{bp}. Since $I$ is an open subgroup of $K$,
$J_0|_I$ is an injective object in $\Mod^{\mathrm{sm}}_{I}(A)$ and thus there exists an idempotent 
$e\in \End_{A[I]}(J_0)$ such that $e\circ \iota_0= \iota_0$ and $\iota_0: V\hookrightarrow e J_0$ is 
an injective envelope of $V$ in $\Mod^{\mathrm{sm}}_{I}(A)$. By Lemma~\ref{A1} there exists an 
isomorphism $\beta: e J_0\overset{\cong}{\rightarrow} (e J_0)^{\Pi}$ extending the given isomorphism 
$\alpha: V\overset{\cong}{\rightarrow} V^{\Pi}$. Lemma~\ref{A2} implies that 
\stepcounter{subsection}
\begin{equation}\label{mult1}
\dim_k \Hom_I( \chi, e J_0[\mm]^{I_1})= \dim_k \Hom_I(\chi^{\Pi}, e J_0[\mm]^{I_1}),\quad  \forall \chi\in \Irr_I(k).
\end{equation}
Since the order of $I/I_1$ is prime to $p$, Lemma~\ref{A3} combined with \eqref{mult1} implies  
\stepcounter{subsection}
\begin{multline}
\dim_k \Hom_I( \chi, (1-e)J_0[\mm]^{I_1}) \\
= \dim_k \Hom_I(\chi^{\Pi}, (1-e)J_0[\mm]^{I_1}),\quad  \forall \chi\in \Irr_I(k).
\end{multline} 
Thus Lemma~\ref{A2} implies that there exists an $I$-equivariant isomorphism $\gamma: 
(1-e) J_0\overset{\cong}{\rightarrow} ((1-e) J_0)^{\Pi}$. Letting $\delta=\beta\oplus \gamma: J_0\overset{\cong}{\rightarrow} J_0^{\Pi}$, 
we obtain a commutative diagram of $A[I]$-modules:
\begin{displaymath}
\xymatrix@1{\;V\;\ar[d]_{\cong}^{\alpha}\ar@{^(->}[r]^-{\iota_0} & J_0\ar[d]^{\delta}_{\cong}\\
              \;V^{\Pi}\; \ar@{^{(}->}[r]^-{\iota_0^{\Pi}} &  J_0^{\Pi}.}
\end{displaymath} 

Since $\iota_0^{\Pi}: V^{\Pi}\hookrightarrow J_0^{\Pi}$ is an injective envelope of $V^{\Pi}$ in
$\Mod^{\mathrm{sm}}_{K^{\Pi}}(A)$, and injective envelopes are unique up to isomorphism, there exists 
a commutative diagram of $A[K^{\Pi}]$-modules: 
\begin{displaymath}
\xymatrix@1{\;V^{\Pi}\;\ar[d]_{\cong}^{\alpha^{-1}}\ar@{^(->}[r]^-{\iota_0^{\Pi}} & J_0^{\Pi}\ar[d]^{\psi}_{\cong}\\
            \;V\; \ar@{^{(}->}[r]^-{\iota_1} &  J_1.}
\end{displaymath} 
Letting $\varphi=\psi\circ \delta$ proves the claim. 
\end{proof}

\begin{remar}\label{centralrem}{\em
The proof of Theorem~\ref{input} works in any reasonable subcategory of $\Mod^{\mathrm{l.adm}}_{G^0}(A)$. 
For example if we fix a smooth  character $\zeta: Z\rightarrow A^{\times}$ and rework the proof of Theorem~\ref{input} by 
considering only objects with central character $\zeta$ we obtain the same statement with $\Mod^{\mathrm{l.adm}}_{G^0}(A)$ replaced by 
$\Mod^{\mathrm{l.adm}}_{G^0, \zeta}(A)$ and $\Mod^{\mathrm{sm}}_K(A)$ replaced by $\Mod^{\mathrm{sm}}_{K, \zeta}(A)$. 
}
\end{remar}

\begin{cor}\label{output1} If $V$ is an injective object in $\Mod^{\mathrm{l.adm}}_{G^0}(A)$, then $V$ is also an injective object 
in $\Mod^{\mathrm{sm}}_{K}(A)$. 
\end{cor} 
\begin{proof} It is enough to show that $V$ is a direct summand of an object which is injective in $\Mod^{\mathrm{sm}}_{K}(A)$. 
By replacing $V$ with $V\oplus V^{\Pi}$ we may assume that there exists $\psi\in \End_{A}(V)$ such that $\psi^2=1$ and 
$\psi\circ g= g^{\Pi}\circ \psi$ for all $g\in G^0$.

Let $\mathcal A$ be the set of admissible subrepresentations of $V$. The set $\mathcal A$ is naturally ordered by inclusion. Moreover, 
it is filtered, since if $U_1, U_2\in \mathcal A$ then $U_1+U_2$ is a quotient of an admissible representation $U_1\oplus U_2$, 
and hence is admissible, see 
\cite{em1} Proposition 2.2.10. Hence, we have an injection 
$$\underset{\overset{\longrightarrow}{U\in \mathcal A}}{\lim} \ U\hookrightarrow V.$$  
Since $V$ is locally admissible every $v\in V$ is contained in some admissible subrepresentation $U$, hence the map is surjective. 
Let $\mathcal I$ be a subset of $\mathcal A$ consisting of those $U$ such that $U|_K$ is an injective object in $\Mod^{\mathrm{sm}}_K(A)$. 
We claim that $\mathcal I$ is cofinal in~$\mathcal A$. To see this,
choose $U\in \mathcal A$.  After replacing $U$ by $U+\psi(U)$ we may assume that 
$U=\psi(U)$ and, in particular, that $\psi$ induces an isomorphism $U\cong U^{\Pi}$. Let $U\hookrightarrow \Omega$ be 
as in Theorem~\ref{input}. Since $V$ is injective and $\Omega$ is admissible there exists $\varphi: \Omega\rightarrow V$ making the following diagram of $G^0$-representations commute:
\begin{displaymath} 
\xymatrix@1{ \;U\;\ar@{^(->}[r]\ar@{^(->}[dr]& V\\ & \Omega\ar[u]_{\varphi}}
\end{displaymath}
Since $U|_K\hookrightarrow \Omega|_K$ is an injective envelope of $U|_K$ we deduce that $\varphi$ is an injection. Since $\varphi(\Omega)$ lies in 
$\mathcal I$ we obtain the claim.  Hence, we obtain an isomorphism
$$\underset{\overset{\longrightarrow}{\Omega\in \mathcal I}}{\lim} \ \Omega \cong V.$$
Since $V|_K$ is an inductive limit of injective objects, \cite{em2} Proposition 2.1.3 implies that $V|_K$ is an injective object in
$\Mod^{\mathrm{sm}}_{K}(A)$.
\end{proof}

\begin{cor}\label{output2} If $V$ is an injective object in $\Mod^{\mathrm{l.adm}}_{G^+}(A)$ then $V$ is also an injective object 
in $\Mod^{\mathrm{sm}}_{K}(A)$. 
\end{cor}  
\begin{proof} We consider $\unif$ as an element of $Z$ via $F^{\times}\cong Z$ and note that $G^+=G^0 \unif^{\mathbb Z}$.
%It is enough to show that if $V$ is injective in $\Mod^{\mathrm{l.adm}}_{G^+}(A)$ then $V$ is also injective 
%in  $\Mod^{\mathrm{l.adm}}_{G^0}(A)$, since then we may appeal to Corollary~\ref{output1}.
%
Let $B = A[ \unif^{\pm 1}] \iso A[t^{\pm 1}].$ If $U$ is any locally admissible $G^+$-representation,
then $U = \bigoplus_{\mathfrak n} U_{\mathfrak n},$
where $\mathfrak n$ runs over the maximal ideals of $B$ and $U_{\mathfrak n}$
denotes the localization of $U$ at $\mathfrak n.$ Furthermore,
$$U_{\mathfrak n} = U[\mathfrak n^{\infty}] := \bigcup_{i \geq 1} U[\mathfrak n^i],$$
where $U[\mathfrak n^i]$
denotes the subspace of $U$ consisting of elements
annihilated by $\mathfrak n^i$.  
Each maximal ideal $\mathfrak n$ is of the form $(\mathfrak m, f)$,
where $\mathfrak m$ is the maximal ideal of $A$, and $f \in A[t]$ is
a monic polynomial.  Since $A$ is Artinian, so that $\mathfrak m$ is a nilpotent
ideal, we see that the $\mathfrak n$-adic topology and $f$-adic topology on
$A$ coincide. 
Thus we may equally well write
$$U_{\mathfrak n} = \bigcup_{i \geq 1} U[f^i],$$
where of course $U[f^i]$
denotes the subspace of $U$ consisting of elements
annihilated by $f^i$.   

Suppose now that $V$ is an injective object of $\Mod^{\mathrm{l.adm}}_{G^+}(A)$.
Since, by the discussion of the preceding paragraph,
$V$ is the inductive limit of the
$V[f^i]$ (where $f^i$ runs over the various powers of the various
monic polynomials associated to the various maximal ideals $\mathfrak n$
of $B$), 
in order to show that $V$ is injective as an object of
$\smKmod(A)$,
%$\Mod^{\mathrm{l.adm}}_{G^0}(A)$,
it suffices, by \cite[Prop.~2.1.3]{em2} together with Corollary~\ref{output1},
to show that
each $V[f^i]$ is an injective object of $\Mod^{\mathrm{l.adm}}_{G^0}(A)$.

If we write $C := B/(f^i)$ then the category 
$\Mod^{\mathrm{l.adm}}_{G^{+}}(A)[f^i]$ is naturally equivalent 
to the category $\Mod^{\mathrm{l.adm}}_{G^0}(C)$.  
Since $f$ is monic $C$ is free of finite rank over $A$ and hence 
the forgetful functor $\Mod^{\mathrm{l.adm}}_{G^0}(C)\rightarrow \Mod^{\mathrm{l.adm}}_{G^0}(A)$
is right adjoint to the exact functor $C\otimes_A \text{--}$, 
and so preserves injectives. Thus $V[f^i]$ is injective in $\Mod^{\mathrm{l.adm}}_{G^0}(A)$.
\end{proof}

\begin{cor}\label{output3} If $V$ is an injective object in $\Mod^{\mathrm{l.adm}}_{G}(A)$ then 
$V$ is also injective in $\Mod^{\mathrm{sm}}_K(A)$.
\end{cor}
\begin{proof} Since $G^+$ is open of finite index in $G$,
Lemma~\ref{lem:adm} and \cite[Prop.~2.1.2]{em2} show that $V$ is injective in $\Mod^{\mathrm{l.adm}}_{G^+}(A)$ and 
the assertion follows from Corollary~\ref{output2}.
\end{proof}

\begin{cor}\label{rest} If $V$ is injective in either of the categories
$\locadmGmodz(A)$ {\em (}for some smooth character $\zeta:Z \rightarrow A^{\times}${\em )},
or  $\locadmGmod(A)$,
then $V|_{N_0}$ is an injective object in $\Mod^{\mathrm{sm}}_{N_0}(A)$.
\end{cor}
\begin{proof}
In the latter case, the claim of the present corollary follows from Corollary ~\ref{output3} together with 
\cite[Prop.~2.1.11]{em2}.
In the case of $\locadmGmodz(A)$, it follows from Remark~\ref{centralrem} and Lemma~\ref{centr} 
that $V$ is injective in $\Mod^{\mathrm{sm}}_{(K\cap Z) N_0, \zeta}(A)$. Since the intersection of 
$N_0$ and $Z$ is trivial restriction to $N_0$ induces an equivalence of categories between 
$\Mod^{\mathrm{sm}}_{(K\cap Z) N_0, \zeta}(A)$ and $\Mod^{\mathrm{sm}}_{N_0}(A)$. Thus $V$ is injective in $\Mod^{\mathrm{sm}}_{N_0}(A)$.
\end{proof}

Let $G$ be the group of $\Qp$-valued points of a connected reductive linear algebraic group over $\Qp$. Let $P$ be a parabolic subgroup 
of $G$ with a Levy subgroup $M$ and let $\overline{P}$ be the parabolic subgroup of $G$ opposite to $P$ with respect to $M$. In \cite{em1},
the first author
defined a left exact functor $\Ord_P: \locadmGmod(A)\rightarrow \locadmMmod(A)$ such that for all $U$ in $\locadmMmod(A)$ 
and $V$ in $\locadmGmod(A)$ one has 
$$\Hom_G(\Indu{\overline{P}}{G}{U}, V)\cong \Hom_{M}(U, \Ord_P(V)).$$  
Further, for $i\ge 0$ in \cite{em2} there are defined functors $H^i\Ord_P:\locadmGmod(A)\rightarrow \locadmMmod(A)$ such that 
$H^0\Ord_P= \Ord_P$ and $\{H^i\Ord_P: i\ge 0\}$ is a $\delta$-functor. It is
conjectured there that for $i\ge 1$ the functors $H^i\Ord_P$ are effaceable, which 
would imply that they are universal, and hence coincide with the derived functors of $\Ord_P$. 

\begin{cor}\label{eff} If $G=\GL_2(F)$ and if $V$ is an injective object in 
$\locadmGmod(A)$ {\em (}resp.\ $\locadmGmodz(A)${\em )},
then $H^i\Ord_P(V)=0$ for all $i\ge 1$.
\end{cor}   
\begin{proof} Since by Corollary~\ref{rest}, $V|_{N_0}$ is an injective object in $\Mod^{\mathrm{sm}}_{N_0}(A)$ 
we have that $H^i(N_0, V)=0$ for all $i\ge 1$. The claim 
follows from the definition of $H^i\Ord_P$, see \cite[Def.3.3.1]{em2}.
\end{proof}

Since  $\locadmGmod(A)$ and  $\locadmGmodz(A)$
each have enough injectives, we conclude that the $H^i\Ord_P$ are effaceable for $i\ge 1$ 
on any of these categories.  In particular, we have verified
\cite[Conj.~3.7.2]{em2} in the case $G = \GL_2(F)$.

\begin{remark}
{\em
The authors of this note strongly believe that an analogue
Theorem~\ref{input} holds for other groups than $\GL_2(F)$.
If this is the case than our proof should go through 
to establish \cite[Conj.~3.7.2]{em2} for these groups.    
}
\end{remark}

\end{document}